\documentclass[12pt,a4paper,reqno]{amsart}
\usepackage{amssymb,amsmath}
\usepackage{amsfonts,amsbsy,bm}
\usepackage{latexsym}
\usepackage{exscale}
\usepackage{hyperref}

\usepackage{color}

\usepackage[normalem]{ulem}

\newcommand{\R}{\mathbb R}
\newcommand{\Z}{\mathbb Z}
\newcommand{\N}{\mathbb N}

\renewcommand{\phi}{{\varphi}}

\def\supp{\mathop{\rm supp}}

\numberwithin{equation}{section}
\newtheorem{theorem}{Theorem}[section]
\newtheorem{lemma}[theorem]{Lemma}

\newtheorem{corollary}[theorem]{Corollary}
\newtheorem{Remark}[theorem]{Remark}
\newtheorem{proposition}[theorem]{Proposition}

\newtheorem{example}[theorem]{Example}

\newcommand{\Ba}[1]{\begin{array}{#1}}
\newcommand{\Ea}{\end{array}}
\newcommand{\Be}{\begin{equation}}
\newcommand{\Ee}{\end{equation}}
\newcommand{\Bea}{\begin{eqnarray}}
\newcommand{\Eea}{\end{eqnarray}}
\newcommand{\Beas}{\begin{eqnarray*}}
\newcommand{\Eeas}{\end{eqnarray*}}
\newcommand{\Benu}{\begin{enumerate}}
\newcommand{\Eenu}{\end{enumerate}}
\newcommand{\Bi}{\begin{itemize}}
\newcommand{\Ei}{\end{itemize}}

\newcommand{\BR}{\begin{Remark} \em}
\newcommand{\ER}{\end{Remark}}
\newcommand{\BE}{\begin{example} \em}
\newcommand{\EE}{\end{example}}

\newcounter{reg}
\newcounter{regTO}
\newcommand{\vertiii}[1]{{\left\vert\kern-0.25ex\left\vert\kern-0.25ex\left\vert #1
		\right\vert\kern-0.25ex\right\vert\kern-0.25ex\right\vert}}

\begin{document}

\title[Extra invariance and the Zak transform]{Extra invariance of principal shift invariant spaces and the Zak transform.}

\author{Davide Barbieri}

\address{Davide Barbieri
\\
Departamento de Matem\'aticas\\
Universidad Aut\'onoma de Madrid\\
28049, Madrid, Spain} \email{davide.barbieri@uam.es}

\author{Eugenio Hern\'andez}

\address{Eugenio Hern\'andez
	\\
	Departamento de Matem\'aticas\\
	Universidad Aut\'onoma de Madrid\\
	28049, Madrid, Spain} 
	 \email{eugenio.hernandez@uam.es}

\author{Carolina Mosquera}
\address{Carolina Mosquera
\\
Departamento de Matem\'atica \\
Universidad de Buenos Aires\\
Ciudad Universitaria, Pabell\'on I, 1428 Buenos Aires, Argentina \\
and IMAS-CONICET, Consejo Nacional de Investigaciones Cient\'ificas y T\'ecnicas, Argentina
} \email{mosquera@dm.uba.ar}

\begin{abstract}
We prove a necessary and sufficient condition for a principal shift invariant space of $L^2(\R)$ to be invariant under translations by the subgroup $\frac{1}{N} \Z\,, N>1$. This condition is given in terms of the Zak transform of the group $\frac{1}{N} \Z\,.$ This result is extended to principal shift invariant spaces generated by a lattice in a general locally compact abelian (LCA) group.
\end{abstract}


\date{\today}


\maketitle

\section{Introduction and main result}\label{secIntroduc}

Let $H$ be an additive subgroup of $\R$. A closed subspace $V$ of $L^2(\R)$ is called $H$-{\bf invariant} if it is invariant under translations by elements of $H$. That is, when $f\in V$, then $T_h(f)\in V$ for all $h\in H$, where $T_h(x)=f(x-h).$ A $\Z$ invariant subspace of $L^2(\R)$ is called {\bf shift invariant}. 

\

Shift invariant spaces are the core spaces of Multiresolution Analysis (\cite{Mey90, Dau92, HW96, Mal99}), and as such they are used to study signals and images.  They are also used as models to approximate functional data (\cite{dBDVR94, ACHM07}).

\

Shift invariant spaces are also the natural spaces for sampling. For a measurable set $A \subset \R$, the Paley-Wiener space $PW(A)$ is defined by
$$
PW(A) := \{ f\in L^2(\R) : \supp \widehat{f} \subset A  \}\,,
$$
where $\displaystyle \widehat f (\xi) = \int_{\R} f(x) e^{-2\pi i x\xi}\,dx\,$ denotes the Fourier transform of $f$. The Whittaker-Shannon-Kotel'nikov sampling theorem establishes that any signal $ f$ in the space $PW([-M/2\,, M/2]), M>0,$ can be recovered with the samples $\{ f(k/M)  \}_{k\in \Z}$ by the formula
\begin{equation} \label{Eq:sampling}
f(x) = \sum_{k\in \Z} f\left( \frac{k}{M}\right)  \frac{\sin \pi(Mx- k)}{\pi(Mx-k)}\,,
\end{equation}
with convergence in $L^2(\R)$ and pointwise uniformly. The space $PW(A)$ is shift invariant, that is invariant under translations by the group $\Z$. It has extra invariance, since it si also invariant under the elements of the group $\R$, a bigger group than $\Z.$

\

There are other closed additive subgroups of $\R$ that contain $\Z$. All of them are of the form $\frac{1}{N} \Z$ for some natural number $N>1$. In \cite{ACHKM10} several equivalent conditions are given to determine if a shift invariant space is also $\frac{1}{N} \Z$ invariant, $N\in \N, N >1.$ Their results are given in terms of cut-off spaces in the Fourier transform side, gramians and range functions.

\

A particular important class of $H$ invariant spaces is the one whose elements are generated by the $H$-translations of a single function $\psi \in L^2(\R)$. They are called  {\bf principal} and define as 
$$
\langle \psi \rangle_{H} := \overline{{\mbox span}} \{T_h \psi: h\in H\},
$$ 
where the closure is taken in $L^2(\R).$ It can be seen using \eqref{Eq:sampling} that the Paley-Wiener space $PW([\-1/2\,, 1/2])$ is principal and generated by the function $\psi \in L^2(\R)$ given by $\widehat \psi = \chi_{[-1/2\,, 1/2]}.$

\

For principal shift invariant spaces $\langle \psi \rangle_\Z \subset L^2(\R)$ it is shown in \cite{SW11} that $\langle \psi \rangle_\Z$ is also $\frac{1}{N} \Z$ invariant, $N\in \N, N >1\,,$ if and only if for all $p=1,2, \dots, N-1\,,$
$$
P_{\psi,N}(\xi) P_{\psi,N}(\xi + p)=0\,, \ a.e. \ \xi\in \R\,,
$$
where $P_{\psi,N}$ is the periodization function of $\psi$ for the group  $\frac{1}{N} \Z\,,$ that is
$$
P_{\psi,N} (\xi) := \sum_{k\in \Z} |\widehat \psi (\xi + Nk)|^2\,.
$$

The Zak transform of a function $f \in L^1(\R)$ for the group $\frac{1}{N} \Z, N\in \N\,,$ is given by 
\begin{equation} \label{Eq:ZakN}
Z_N(f)(x,\xi) := \frac{1}{N} \sum_{k\in \Z} f(x + \frac{k}{N}) e^{-2\pi i \frac{k}{N}\xi}\,, \ x, \xi \in \R\,.
\end{equation}
It can be extended to be an isometric isomorphism from $L^2(\R)$ onto $L^2([0,1/N)\times [0,N)).$ For the proof of this result and other properties of the Zak transform, together with historical background and references, see \cite{Jan88}. It is a very useful tool in time-frequency analysis (\cite{Gro01}) and in situations where the Fourier transform is not available, such as in Harmonic Analysis in non-commutative discrete groups (\cite{BHP14, BHP15}). And it is also of great value in abelian Fourier Analysis: as an example see the simple proof of the Plancherel Theorem given in \cite{HSWW10} using the Zak transform.

\

The main purpose of this article is to give a characterization of $\frac{1}{N} \Z$ extra invariant of principal shift invariant spaces of $L^2(\R)$ using the Zak transform of the group $\frac{1}{N} \Z$. The statement is the following:

\begin{theorem} \label{Th:1-1}
	Let $\psi\neq 0, \psi\in L^2(\R)$ and $N\in \N, N>1.$ The following are equivalent:
	
	(a) $\langle \psi \rangle_\Z$ is $\frac{1}{N} \Z$ invariant.
	
	(b) $Z_N(\psi)(x, \xi + p)\, Z_N(\psi)(y, \xi+q) = 0$ a.e. $x,y \in [0\,, 1/N)$, a.e. $\xi \in [0, 1)$, for all
	$p, q = 0, 1, \dots N-1, p \neq q\,.$ 
\end{theorem}

Let $I_1=[0,1/2)\cup[1,3/2)$ and $I_2=[0,1/2)\cup[3/2,2).$ The functions $\psi_1$ and $\psi_2$ given by $\widehat{\psi_1} = \chi_{I_1}$ and $\widehat{\psi_2} = \chi_{I_2}$ are exhibited in \cite{SW11} to show that although both allow sampling formulas with the lattice $\frac{1}{2} \Z$, the second one is better with respect to sampling since it also allows sampling with the coarser lattice $\Z,$ while the first one does not.

\

We can witness, using Theorem \ref{Th:1-1}, that they are also different for $\frac{1}{2} \Z$ extra invariance. To see this, we borrow from Proposition \ref{Pro:2-3} the following formula for the Zak transform of a function $f \in L^1(\R)$ for the group $\frac{1}{N} \Z, N\in \N:$
\begin{equation} \label{Eq_:ZakN}
    Z_N(f)(x,\xi) = \sum_{k\in \Z} \widehat f (\xi + Nk) e^{2\pi i x \cdot (\xi + Nk) }\,, \quad x,\ \xi \in \R\,.
\end{equation}
Using \eqref{Eq_:ZakN} it is easy to see that for $\xi \in [0,1)\,,$ and $x, y \in \R\,,$
$$
Z_2(\psi_1)(x, \xi) = e^{2\pi i x\cdot \xi} \chi_{[0,1/2)}(\xi)\,, \qquad \mbox{and} \qquad Z_2(\psi_1)(y, \xi + 1) = e^{2\pi i x\cdot (\xi + 1)} \chi_{[0,1/2)}(\xi)\,.
$$
Since $Z_2(\psi_1)(x, \xi) Z_2(\psi_1)(y, \xi + 1) \neq 0$ for all $x,y\in \R$ and all $\xi \in [0, 1/2)\,,$ we deduce from Theorem \ref{Th:1-1} that $\langle \psi_1 \rangle_{\Z}$ is not $\frac12 \Z$ invariant. On the other hand, using again \eqref{Eq_:ZakN}, for all $\xi \in [0,1)\,,$ and $x, y \in \R\,,$ we have
$$
Z_2(\psi_2)(x, \xi) = e^{2\pi i x\cdot \xi} \chi_{[0,1/2)}(\xi)\,, \qquad \mbox{and} \qquad Z_2(\psi_2)(y, \xi + 1) = e^{2\pi i x\cdot (\xi + 1)} \chi_{[1/2,1)}(\xi)\,
$$
Hence, $Z_2(\psi_2)(x, \xi) Z_2(\psi_2)(y, \xi + 1) = 0$ for all $\xi \in [0,1)\,,$ and $x, y \in \R\,.$ This shows, by Theorem \ref{Th:1-1}, that $\langle \psi_2 \rangle_{\Z}$ is  $\frac12 \Z$ invariant.

\

Theorem 1.1 will be proved in Section \ref{Proof}. Section \ref{Preliminaries} contains the tools needed for the proof. In Section \ref{ProofLCA} we generalize Theorem 1.1 to the case of locally compact abelian (LCA) groups. We need an LCA group $G$ and two lattices $\mathcal K \subset \mathcal L$ of $G$. The dual group of $G$ will be denoted by $\widehat G$ and $\mathcal L^\perp \subset \mathcal K^\perp$ denote the dual lattices of $\mathcal L$ and $\mathcal K$ respectively. We denote by $C_{\mathcal L}$ a measurable tiling set of $G$ by $\mathcal L$, and similarly by $C_{\mathcal K^\perp}$ a measurable tiling set of $\widehat G$ by $\mathcal  K^\perp$. 

We also need the notion of Zak transform with respect to a lattice that the reader can find in \eqref{Def:4-5}.

\begin{theorem} \label{Th:1-2}
	Let $G$ be an LCA group and let $\mathcal K \subset \mathcal L$ be two lattices in $G$. Let $\psi\neq 0, \psi\in L^2(G)$. The following are equivalent:
	
	(a) $\langle \psi \rangle_{\mathcal K}$ is $\mathcal L$ invariant.
	
	(b) $Z_{\mathcal L}(\psi)(\alpha + \beta_1, x) \, Z_{\mathcal L}(\psi)(\alpha + \beta_2, y) = 0$ for all $\beta_1, \beta_2$ such that $[\beta_1] \neq [\beta_2]$ in $\mathcal K^\perp /\mathcal L^\perp$, and a.e. $x,y \in C_{\mathcal L},\, \alpha \in C_{\mathcal K^\perp}.$
\end{theorem}

The proof of Theorem \ref{Th:1-2} will be given in Section \ref{ProofLCA}. Subsections \ref{Background}, \ref{ZakLCA}, and \ref{PrincipalLCA} contain the tools needed for the proof.

\
 
{\small {\bf Acknowledgements}. This project has received funding from the European Union's Horizon 2020 research and innovation programme under the Marie Sklodowska-Curie grant agreement No 777822. D. Barbieri and E.
Hern\'andez were supported by Grant MTM2016-76566-P (Spain). C. Mosquera was  supported by Grants UBACyT 20020170100430BA, PICT 2014--1480 and PIP 112201501003553CO (Argentina). E. Hern\'andez would like to thank the Isaac Newton Institute of Mathematical Sciences for support and hospitality during the programme \emph {Approximation, sampling and compression in data science}, when part of this work was undertaken. This programme was supported by EPSRC Grant Number EP/R014604/1.

We thank Carlos Cabrelli and Victoria Paternostro for reading the manuscript and providing some interesting comments.

\newpage

\section{Preliminaries} \label{Preliminaries}

\subsection{Properties of the Zak transform.} \label{Zak}

Recall from \eqref{Eq:ZakN} that the Zak transform of a function $f \in L^1(\R)$ for the group $\frac{1}{N} \Z, N\in \N\,,$ is given by 
\begin{equation} \label{Eq:ZakN2}
Z_N(f)(x,\xi) := \frac{1}{N} \sum_{k\in \Z} f(x + \frac{k}{N}) e^{-2\pi i \frac{k}{N}\xi}\,, \ x, \xi \in \R\,.
\end{equation}
It can be extended to be an isometric isomorphism from $L^2(\R)$ onto $L^2([0,1/N)\times [0,N)).$ It follows from the definition that if $\ell \in \Z$ and $x, \xi \in \R$
\begin{equation} \label{Eq:property1}
Z_N (f)(x , \xi +\ell N) = Z_N (f)(x, \xi)\,,
\end{equation}
and
\begin{equation} \label{Eq:property2}
Z_N (f)(x + \frac{\ell}{N} \,, \xi) = e^{2 \pi i \frac{\ell}{N} \xi }Z_N(f) (x, \xi)\,.
\end{equation}
Therefore, $Z_N (f)(x, \xi)$ is determined as soon as we know its values in the rectangle $[0,1/N)\times [0,N).$

\

The following result relates the usual Zak transform $Z_1$ with the Zak transform defined by \eqref{Eq:ZakN2}.

\begin{proposition} \label{Pro:2-1}
	For $f\in L^2(\R)$, $x \,, \xi \in \R$, and $N \in \N$,
	$$
	Z_1(f)(x,\xi) = \sum_{q=0}^{N-1} Z_N(f)(x, \xi +q )\,.
	$$
\end{proposition}

\begin{proof}
	By density, it is enough to prove the result for $f\in L^1(\R)\cap L^2(\R).$ Using definition \eqref{Eq:ZakN2} and collecting terms we obtain
	\begin{eqnarray*}
		\sum_{q=0}^{N-1} Z_N(f)(x, \xi +q ) &=&
		\sum_{q=0}^{N-1} \frac{1}{N} \sum_{k\in \Z} f(x + \frac{k}{N}) e^{-2\pi i \frac{k}{N}(\xi+ q)} \\
		&=& \sum_{k\in \Z} \frac{1}{N} \left(\sum_{q=0}^{N-1} e^{-2\pi i \frac{k}{N}q}  \right) f(x + \frac{k}{N}) e^{-2\pi i \frac{k}{N}\xi}\,.
	\end{eqnarray*}
Let $\displaystyle \Phi_N(k) := \frac{1}{N} \left(\sum_{q=0}^{N-1} e^{-2\pi i \frac{k}{N}q}  \right)$. If $K=\ell N$, $\Phi_N(k)=1.$ On the other hand if $k$ is not an integer multiple of $N$, using the sum of a geometric progression, $\Phi_{N}(k) = 0.$ Therefore,
$$
\sum_{q=0}^{N-1} Z_N(f)(x, \xi +q ) = \sum_{\ell\in \Z} f(x+\ell) e^{-2\pi i \ell \xi} = Z_1(f)(x,\xi).
$$
\end{proof}

Recall that $T_x(f)(y) = f(y-x)$ denotes de translation by $x\in \R.$
\begin{proposition} \label{Pro:2-2}
	For $f\in L^2(\R)$, $x \,, \xi \in \R$, and $N \in \N$,
	$$
	Z_1(T_{1/N}(f))(x,\xi) = \sum_{q=0}^{N-1} e^{-\frac{2 \pi i (\xi + q)}{N}} Z_N(f)(x, \xi +q )\,.
	$$
\end{proposition}

\begin{proof}
	By density, it is enough to prove the result for $f\in L^1(\R)\cap L^2(\R).$ Now, the result follows by using Proposition \ref{Pro:2-1} and equation \eqref{Eq:property2}:
	\begin{eqnarray*}
		Z_1(T_{1/N}(f)) (x, \xi)&=& Z_1(f)(x-\frac1N, \xi) = 
		\sum_{q=0}^{N-1} Z_N(f)(x -\frac1N, \xi +q ) \\
		&=& \sum_{q=0}^{N-1} e^{- 2\pi i\frac1N (\xi + q)}Z_N(f) (x, \xi+q)\,.
	\end{eqnarray*}
\end{proof}

The following result will be needed in the sequel. It gives a way to compute the Zak transform of a function using its Fourier transform.
\begin{proposition} \label{Pro:2-3}
	For $f\in L^2(\R)$, $x \,, \xi \in \R$, and $N \in \N$,
	$$
	Z_N(f)(x,\xi) = \sum_{k\in \Z} \widehat f (\xi + Nk) e^{2\pi i x \cdot (\xi + Nk)}\,.
	$$
\end{proposition}

\begin{proof}
	It is enough to show the result for $f\in C_c(\R)$, the continuous functions with compact support in $\R$.
	For each $x,\xi \in \R$ define
	$\displaystyle F_{x,\xi}(t) := f(x + \frac{t}{N}) e^{-2\pi i \frac{t}{N}\xi}\,, \  t \in \R\,.$ By the Poisson Summation Formula,
	\begin{equation} \label{Eq:2-4}
	Z_N(f)(x,\xi) = \frac1N \sum_{k\in \Z} F_{x,\xi}(k) = \frac1N \sum_{k\in \Z} \widehat{F_{x,\xi}}(k)\,.
	\end{equation}
	With the change of variables $x+\frac{t}{N} =z$ we obtain
	\begin{eqnarray*}
		\widehat{F_{x,\xi}}(k) &=& \int_{\R} F_{x,\xi}(t) e^{-2\pi i k t} \, dt = \int_{\R} f(x + \frac{t}{N}) e^{-2\pi i \frac{t}{N}\xi}\, e^{-2\pi i k t} \, dt \\
		&=& N \int_{\R} f(z)\, e^{-2\pi i (z - x)\xi} e^{-2\pi i k (z-x)N}\,dz \\
		&=& N e^{2\pi i x \dot (\xi + Nk)} \int_{\R} f(z)\, e^{-2\pi i z \cdot (\xi + Nk)}\,dz \\
		&=& N e^{2\pi i x \dot (\xi + Nk)} \widehat f(\xi + Nk)\,.
	\end{eqnarray*}
    The result follows by replacing this equality in \eqref{Eq:2-4}.
\end{proof}

\begin{Remark}
	Propositions \ref{Pro:2-1} and \ref{Pro:2-2} can also be proved using Proposition \ref{Pro:2-3}. We leave the details for the reader.
\end{Remark}

\subsection{Principal invariant spaces.}  \label{Invariant}

We start by giving a condition to determine if a principal shift invariant subspace is also $\frac1N \Z$ invariant.

\begin{proposition} \label{Pro:2-5}
	Let $\psi \neq 0, \psi \in L^2(\R),$ and $N\in \N, N > 1.$ The following are equivalent:
	
	(a) $\langle \psi \rangle_\Z$ is $\frac1N \Z$ invariant.
	
	(b) $T_{1/N}(\psi) \in \langle \psi \rangle_\Z.$
\end{proposition}

\begin{proof}
	$(a) \Rightarrow (b)$ is clear by definition. To prove $(b) \Rightarrow (a)$ let $f\in \langle \psi \rangle_\Z.$ We have to show that $T_{k/N}(f) \in \langle \psi \rangle_\Z$ for all $k\in \Z.$ Write $k = \ell N + q, \ell \in \Z, q\in \Z, 0 \leq q \leq N-1.$ Then, $T_{k/N} (f) = T_{q/N}T_\ell (f) \in T_{q/N}(\langle \psi \rangle_\Z).$ Thus, it is enough to prove that $T_{q/N}(\langle \psi \rangle_\Z) \subset \langle \psi \rangle_\Z$ for all $q\in \Z, 0 \leq q \leq N-1.$ The result is clear for $q=0.$ If $q=1$,
	$$
	T_{1/N} (\langle \psi \rangle_\Z) \subset \langle T_{1/N}(\psi) \rangle_\Z \subset \langle \psi \rangle_\Z\,,
	$$
	since $T_{1/N}\psi \in \langle \psi \rangle_\Z$ by $(b)$. Proceed now by induction on $q$. If $T_{q/N}(\langle \psi \rangle_\Z) \subset \langle \psi \rangle_\Z$, then
	$$
	T_{\frac{q+1}{N}} (\langle \psi \rangle_\Z) = T_{\frac1N}T_{\frac{q}{N}} (\langle \psi \rangle_\Z) \subset T_{\frac1N}(\langle \psi \rangle_\Z) \subset \langle T_{1/N}(\psi) \rangle_\Z \subset \langle \psi \rangle_\Z\,,
	$$
	since $T_{1/N}\psi \in \langle \psi \rangle_\Z$ by $(b)$.
\end{proof}

The following result characterizes the elements of $\langle \psi \rangle_\Z$ in terms of a multiplier. It was first proved in \cite{dBDVR94}, Theorem 2.14 (see also Theorem 2.1 in \cite{HSWW10}).

\begin{proposition} \label{Pro:multiplier}
	Let $\psi \neq 0, \psi \in L^2(\R).$
	
	(a) If $f\in \langle \psi \rangle_\Z$, there exists a $\Z$-periodic function $m_f$ on $\R$ such that $\widehat f = m_f \widehat \psi.$
	
	(b) If $m$ is a $\Z$-periodic function on $\R$ such that $m\widehat \psi \in L^2(\R)$ then, the function $f$ defined by $\widehat f = m\widehat \psi$ belongs to $\langle \psi \rangle_\Z\,.$
\end{proposition}

We will need a similar result to the one stated in the above Proposition, but in terms of multipliers of the Zak transform. It is a corollary of Proposition \ref{Pro:multiplier}.

\begin{corollary} \label{Cor:2-7}
	Let $\psi \neq 0, \psi \in L^2(\R).$
	
	(a) If $f\in \langle \psi \rangle_\Z$, there exists a $\Z$-periodic function $m_f$ on $\R$ with $m_f\widehat \psi \in L^2(\R)$ such that $Z_1(f)(x,\xi) = m_f(\xi) Z_1(\psi)(x,\xi)\,,\ a.\, e. \, x \,, \xi \in \R.$
	
	(b) If $m$ is a $\Z$-periodic function on $\R$ such that $m\widehat \psi \in L^2(\R)$ then, the function $f$ defined by $Z_1(f)(x,\xi) = m(\xi) Z_1(\psi)(x,\xi)\,,\ a.\, e.   \, x \,, \xi \in \R,$ belongs to $\langle \psi \rangle_\Z\,.$
\end{corollary}

\begin{proof}
	$(a)$ By $(a)$ of Proposition \ref{Pro:multiplier} there exists a $\Z$-periodic function $m_f$ on $\R$ such that $\widehat f = m_f \widehat \psi.$ Hence, $m_f \widehat \psi \in L^2(\R)$ and by Proposition \ref{Pro:2-3} for $N=1$ we deduce,
	\begin{eqnarray*}
	   Z_1(f)(x,\xi) &=& \sum_{k\in \Z} \widehat f (\xi + k) e^{2\pi i x \cdot (\xi + k)} = \sum_{k\in \Z} m_f(\xi + k) \widehat \psi (\xi + k) e^{2\pi i x \cdot (\xi + k)}\\
	   &=& m_f(\xi) \sum_{k\in \Z} \widehat \psi (\xi + k) e^{2\pi i x \cdot (\xi + k)} = m_f(\xi) Z_1 \psi(x,\xi)\,.
	\end{eqnarray*}

	   $(b)$ First notice that since $m\widehat \psi \in L^2 (\R)$, by Proposition \ref{Pro:2-3} for $N=1$,
	   \begin{eqnarray*}
	      m(\xi) Z_1(\psi)(x,\xi) &=& m(\xi)\sum_{k\in \Z} \widehat \psi (\xi + k) e^{2\pi i x \cdot (\xi + k)} = \sum_{k\in \Z} m(\xi+k)\widehat \psi (\xi + k) e^{2\pi i x \cdot (\xi + k)}\\
	      &=& \sum_{k\in \Z} m\widehat \psi (\xi + k) e^{2\pi i x \cdot (\xi + k)} = Z_1(\mathcal F^{-1}(m\widehat \psi))(x,\xi)\,,
	   \end{eqnarray*}
	  where $\mathcal F^{-1}$ is our notation for the inverse Fourier transform of an $L^2(\R)$ function. This shows that $m Z_1(\psi)$ coincides with the Zak transform of $\mathcal F^{-1}(m\widehat \psi)) \in L^2(\R) .$ Since $f$ satisfies $Z_1(f)(x,\xi) = m(\xi) Z_1(\psi)(x,\xi) = Z_1(\mathcal F^{-1}(m\widehat \psi))(x,\xi)$ and $Z_1$ is an isometry, we conclude $\widehat f = m \widehat \psi.$  By $(b)$ of Proposition \ref{Pro:multiplier}, $f\in \langle \psi \rangle_\Z\,.$
\end{proof}

\

\section{Proof of Theorem \ref{Th:1-1}} \label{Proof}

\subsection{Proof of (a) implies (b) of Theorem \ref{Th:1-1}}

Assume that $\langle \psi \rangle_\Z$ is $\frac{1}{N} \Z$ invariant, $N\in \N, N>1.$ Then \ref{Pro:2-5}, $T_{1/N}(\psi) \in \langle \psi \rangle_\Z.$ By Corollary \ref{Cor:2-7}, there exists a $\Z$-periodic function $m$ on $\R$, with $m\widehat \psi \in L^2(\R)$, such that 
$$
Z_1(T_{1/N}(\psi))(x,\xi) = m(\xi) Z_1(\psi)(x,\xi)\,, \ a.\, e. \, x \,, \xi \in \R.
$$
Equivalently,
$$
Z_1(\psi)(x-\frac1N,\xi) = m(\xi) Z_1(\psi)(x,\xi)\,,\ a.\, e. \, x \,, \xi \in \R.
$$
Iterating, for $p=0,1,2,\dots$
\begin{equation} \label{Eq:3-1}
Z_1(\psi)(x-\frac{p}{N},\xi) = m(\xi)^p Z_1(\psi)(x,\xi)\,,\ a.\, e. \, x \,, \xi \in \R.
\end{equation}

On the other hand, by Proposition \ref{Pro:2-1} and equation \eqref{Eq:property2}, for $p\in \Z$ we obtain,
\begin{eqnarray}
    Z_1(\psi)(x - \frac{p}{N},\xi) &=& \sum_{q=0}^{N-1} Z_N(\psi)(x - \frac{p}{N}, \xi +q ) \nonumber \\
    &=&  \sum_{q=0}^{N-1} e^{- 2 \pi i \frac{p(\xi + q)}{N} }Z_N(\psi) (x, \xi+q) \nonumber \\
    &=& e^{- 2 \pi i \frac{p \xi}{N}} \sum_{q=0}^{N-1} e^{- 2 \pi i \frac{pq}{N}}Z_N(\psi) (x, \xi+q)\,. \label{Eq:3-2}
\end{eqnarray}

For $q, p \in \Z$ and $x, \xi \in \R$, let $\alpha_q(x,\xi) := Z_N(\psi)(x, \xi+q)$ and $\displaystyle A_p(x,\xi) := \sum_{q=0}^{N-1} e^{- 2 \pi i \frac{pq}{N}} \alpha_q(x,\xi).$
Observe that for $x,\xi$ fixed, $\alpha_q(x,\xi)$ is $N\Z$-periodic in $q$ (see \ref{Eq:property1}). Also $A_p(x,\xi)$ are the discrete Fourier coefficients of the sequence $\displaystyle \{ \alpha_q(x,\xi)  \}_{q=0}^{N-1}.$ Thus, by inversion,
\begin{equation} \label{Eq:3-3}
   \alpha_q(x,\xi) = \frac1N \sum_{p=0}^{N-1}  e^{ 2 \pi i \frac{pq}{N}} A_p(x,\xi)\,, \ x \,, \xi \in \R\,.
\end{equation}

For these coefficients $A_p(x,\xi)$ the following crucial relation can be proved:

\begin{lemma} \label{Lem:3-1}
	Let $x, y, \xi \in \R.$ If $p, q, p_1. q_1 \in \N$ and $p+q = p_1+q_1 (mod N)$, then $$A_p(x,\xi) A_q(y,\xi) = A_{p_1}(x,\xi) A_{q_1}(y,\xi).$$
\end{lemma}

\begin{proof}
	By equation \eqref{Eq:3-2}, 
	$$
	e^{- \frac{2\pi i (p+q)\xi}{N}}A_p(x,\xi) A_q(y,\xi) = Z_1(\psi)(x-\frac{p}{N},\xi) Z_1(\psi)(y-\frac{q}{N},\xi)\,.
	$$
	By equation \eqref{Eq:3-1},
	\begin{equation} \label{Eq:3-4}
		e^{- \frac{2\pi i (p+q)\xi}{N}}A_p(x,\xi) A_q(y,\xi) = m(\xi)^{p+q} Z_1(\psi)(x,\xi) Z_1(\psi)(y,\xi)\,.
	\end{equation}
	Similarly, 
	\begin{equation} \label{Eq:3-5}
	e^{- \frac{2\pi i (p_1+q_1)\xi}{N}}A_{p_1}(x,\xi) A_{q_1}(y,\xi) = m(\xi)^{p_1+q_1} Z_1(\psi)(x,\xi) Z_1(\psi)(y,\xi)\,.
	\end{equation}
	For $k=0,1,2, \dots,$ use \eqref{Eq:3-1} with $p=kN$ and then \eqref{Eq:property2} with $k=\ell$ and $N=1$ to obtain
	\begin{equation} \label{Eq:3-6}
	    m(\xi)^{kN} Z_1(\psi)(x,\xi) = Z_1(\psi)(x-k,\xi) = e^{-2\pi i k \xi} Z_1(\psi)(x,\xi)\,.
	\end{equation}
	
	Assume $p_1 + q_1 = p + q+ kN$ for some $k=0, 1, 2, \dots.$ Then, by \eqref{Eq:3-5}, \eqref{Eq:3-6} and \eqref{Eq:3-4},
	\begin{eqnarray*}
	     A_{p_1}(x,\xi) A_{q_1}(y,\xi) &=& e^{ \frac{2\pi i (p_1+q_1)\xi}{N}} m(\xi)^{p_1+q_1} Z_1(\psi)(x,\xi) Z_1(\psi)(y,\xi) \\
	     &=& e^{ \frac{2\pi i (p+q)\xi}{N}} e^{2\pi i k \xi} m(\xi)^{p+q} m(\xi)^{kN }Z_1(\psi)(x,\xi) Z_1(\psi)(y,\xi) \\
	     &=& e^{ \frac{2\pi i (p+q)\xi}{N}} m(\xi)^{p+q} Z_1(\psi)(x,\xi) Z_1(\psi)(y,\xi) \\
	     &=&  A_{p}(x,\xi) A_{q}(y,\xi)\,.
	\end{eqnarray*}
\end{proof}

We continue now with the proof. With the notation introduced above, we need to show that $\alpha_p(x,\xi) \,\alpha_q(y,\xi) = 0 \,\  a.\, e.\, x,y\in [0,1/N), \, a. \, e. \, \xi \in [0,1),$ for all $p, q = 0, 1, \dots N-1, p \neq q\,.$ 
By equation \eqref{Eq:3-3} 
\begin{eqnarray*}
    \alpha_p(x,\xi) \,\alpha_q(y,\xi) &=& \frac{1}{N^2} \left(\sum_{j=0}^{N-1}  e^{ 2 \pi i \frac{jp}{N}} A_j(x,\xi) \right) \left(\sum_{\ell=0}^{N-1}  e^{ 2 \pi i \frac{\ell q }{N}} A_\ell(y,\xi) \right)\\
    &=& \frac{1}{N^2} \sum_{j=0}^{N-1} \sum_{\ell=0}^{N-1} e^{ 2 \pi i \frac{(jp + \ell q)}{N}} A_j(x,\xi) A_\ell(y,\xi)\,.
\end{eqnarray*}
Let $\ell = N - 1 -j - k.$ Then,
$$
\alpha_p(x,\xi) \,\alpha_q(y,\xi) = \frac{1}{N^2}\sum_{j=0}^{N-1} \sum_{k=-j}^{N-1-j}  e^{\frac{2\pi i j (p-q)}{N}} e^{-\frac{2\pi i (k+1)q}{N}} A_j(x,\xi) A_{N-1-j-k}(y,\xi)\,.
$$
By Lemma \ref{Lem:3-1}, $ A_j(x,\xi) A_{N-1-j-k}(y,\xi) = A_0(x,\xi) A_{N-1-k}(y,\xi).$ Thus, 
$$
\alpha_p(x,\xi) \,\alpha_q(y,\xi) = \frac{1}{N^2} \sum_{j=0}^{N-1} \sum_{k=-j}^{N-1-j}  e^{\frac{2\pi i j (p-q)}{N}} e^{-\frac{2\pi i (k+1)q}{N}} A_0(x,\xi) A_{N-1-k}(y,\xi)\,.
$$
Interchanging, carefully, the above summations, and using that $A_0(x,\xi) A_{2N-1-\ell}(y,\xi) = A_{0}(x,\xi) A_{N-1-\ell}(y,\xi)$ by Lemma \ref{Lem:3-1},  we obtain,
\begin{eqnarray*}
	\alpha_p(x,\xi) \,\alpha_q(y,\xi) &=& \frac{1}{N^2} \sum_{k=0}^{N-1} \sum_{j=0}^{N-1-k}  e^{\frac{2\pi i j (p-q)}{N}} e^{-\frac{2\pi i (k+1)q}{N}} A_0(x,\xi)  A_{N-1-k}(y,\xi) \\
	& & \qquad + \frac{1}{N^2} \sum_{k=-N+1}^{-1} \sum_{j=-k}^{N-1}  e^{\frac{2\pi i j (p-q)}{N}} e^{-\frac{2\pi i (k+1)q}{N}} A_0(x,\xi)  A_{N-1-k}(y,\xi) \\
	&=&  \frac{1}{N^2} \sum_{\ell=0}^{N-1} \sum_{j=0}^{N-1-\ell}  e^{\frac{2\pi i j (p-q)}{N}} e^{-\frac{2\pi i (\ell+1)q}{N}} A_0(x,\xi)  A_{N-1-\ell}(y,\xi) \\
	& & \qquad + \frac{1}{N^2} \sum_{\ell=1}^{N-1} \sum_{j=N-\ell}^{N-1}  e^{\frac{2\pi i j (p-q)}{N}} e^{-\frac{2\pi i (\ell+1)q}{N}} A_0(x,\xi)  A_{2N-1-\ell}(y,\xi) \\
	&=& \frac{1}{N^2} \sum_{\ell=0}^{N-1} \left( \sum_{j=0}^{N-1} e^{\frac{2\pi i j (p-q)}{N}} \right) e^{-\frac{2\pi i (\ell+1)q}{N}} A_0(x,\xi)  A_{N-1-\ell}(y,\xi)\,.
\end{eqnarray*}
Since, when $p\neq q$, $\displaystyle \sum_{j=0}^{N-1} e^{\frac{2\pi i j (p-q)}{N}} = 0$ the result is established.

\subsection{Proof of (b) implies (a) of Theorem \ref{Th:1-1}}

Suppose that 
\begin{equation} \label{Eq:3-7}
     Z_N(\psi)(x, \xi + p)\, Z_N(\psi)(y, \xi+q) = 0
\end{equation}
a.e. $x,y \in [0\,, 1/N)$, a.e. $\xi \in [0, 1)$, for all
$p, q = 0, 1, \dots N-1, p \neq q\,.$ By \eqref{Eq:property2}, equation \eqref{Eq:3-7} holds for a. e. $x,y \in \R.$ By Proposition \ref{Pro:2-5} and Corollary \ref{Cor:2-7} it is enough to find a $\Z$-periodic function $m$ defined on $\R$ such that $m\widehat \psi \in L^2(\R)$ and
\begin{equation} \label{Eq:3-8}
    Z_1(T_{1/N}(\psi))(x,\xi) = m(\xi) Z_1(\psi)(x,\xi)
\end{equation}
a. e. $x, \xi \in \R.$ By the quasi-periodicity properties of $Z_1$ (see \eqref{Eq:property1} and \eqref{Eq:property2}) it is enough to prove \eqref{Eq:3-8} for a. e. $x, \xi \in [0,1).$

For $0 \leq q \leq N-1$ and $ 0\leq x < 1$, let 
$$
   S_\psi^{(q)}(x) := \{ \xi \in [0,1): Z_N(\psi)(x, \xi+q) \neq 0  \}\,,
$$
and
$$
   S_\psi^{(q)} := \bigcup_{x\in [0,1)} S_\psi^{(q)}(x)\,.
$$
Note that $S_\psi^{(q)}$ is a measurable subset of $[0,1)\times [0,1)$. From \eqref{Eq:3-7} we conclude  $ |S_\psi^{(q)} \cap S_\psi^{(p)}| = 0$ when $p,q = 0, 1, 2, \dots, N-1, p \neq q.$ Finally, define 
$$
   S_\psi = [0,1) \setminus \bigcup_{q=0}^{N-1} S_\psi^{(q)}\,.
$$
For $0 \leq \xi < 1$, define
\[
m(\xi)= \begin{cases}
e^{\frac{-2\pi i(\xi+q)}{N}}\, &\mbox{ if }\xi\in S^{(q)}, 0\le q\le N-1\\
1\, &\mbox{ if } \xi \in S_\psi
\end{cases},
\]
and extend $m$ to $\R$ to be $\Z$-periodic. Since $|m(\xi)| = 1$ and $\psi \in L^2(\R)$, we conclude $m\widehat \psi \in L^2(\R).$

We need to show that \eqref{Eq:3-8} holds for a. e. $x, \xi \in [0,1).$ For almost every $x, \xi \in [0,1)$ either $Z_N(\psi)(x,\xi+q)=0$ for all $q=0,1,2, \dots, N-1$ or there exists only one value of $q \in \{0,1,2, \dots, N-1\}$ such that $Z_N(\psi)(x,\xi+q)\neq 0.$ In the first case, by Propositions \ref{Pro:2-1} and \ref{Pro:2-2} we have
$$
Z_1(\psi)(x,\xi) = 0 \qquad \mbox{and} \qquad Z_1(T_{1/N})(x, \xi)= 0\,,
$$
so that \eqref{Eq:3-8} holds trivially. In the second case, again by  Propositions \ref{Pro:2-1} and \ref{Pro:2-2} we have
$$
Z_1(\psi)(x,\xi) = Z_N(\psi)(x, \xi+q) \qquad \mbox{and} \qquad Z_1(T_{1/N})(x, \xi)= e^{- \frac{2 \pi i (\xi + q)}{N}}Z_N(\psi)(x, \xi+q) \,.
$$
Since, in this case, $\xi \in S_\psi^{(q)}(x),$ we have $m(\xi) = e^{- \frac{2 \pi i (\xi + q)}{N}}$ and the equality \eqref{Eq:3-8} also holds in this case.

\section{Tools and results for LCA groups} \label{LCA}

A natural question is to ask if Theorem \ref{Th:1-1} can be extended to locally compact abelian (LCA) groups. 
In \cite{ACP10} the authors characterize the extra invariance of shift invariant spaces on LCA groups in terms of cut--off 
spaces in the Fourier transform side, and also in terms of range functions. Here, we give a characterization using the Zak transform relative to a given lattice.

We start by describing the results we need for our extension. For a detailed introduction to LCA groups see \cite{Rud92}.

\subsection{Background on LCA groups}   \label{Background}

A group $(G, +)$ is an LCA (locally compact abelian) group if it is endowed with a separable, locally compact, Hausdorff topology, the map $x \longrightarrow -x$ is continuous from $G$ into $G$, and the map $(x,y) \longrightarrow x+y$ is continuous from $G\times G$ into $G$. Every LCA group $G$ has a non-zero Borel measure which is translation invariant and unique, up to a possible scalar multiple, called Haar measure, and denoted by $\mu_G$.

\

A character of an LCA group $G$ is a continuous homomorphism $\alpha: G \longrightarrow S^1 = \{z\in \mathbb C : |z|=1  \}.$ The set of all characters of $G$, with the compact open topology, is an LCA group, denoted by $\widehat G$, the {\bf dual group} of $G.$ We write $(x,\alpha) = \alpha (x)$ when $x\in G$ and $\alpha \in \widehat G.$ Notice that for $x,y\in G$ and $\alpha \in \widehat G$, $(x+y,\alpha) = (x,\alpha) (y, \alpha)$ since $\alpha$ is a homomorphism. Thus, $(0,\alpha)=1$, for any $\alpha \in \widehat G$. Similarly, for $x\in G$ and $\alpha, \beta \in \widehat G$, $(x, \alpha + \beta) = (x,\alpha)(x,\beta)$ and $(x,0)=1.$

\

A subgroup $\mathcal L$ fo $G$ is called a {\bf lattice} it is discrete with respect to the topology of $G$ and $T_{\mathcal L} = G/{\mathcal L}$ is compact in the quotient topology. In particular $\mathcal L$ is countable. Associated to a lattice $\mathcal L$ of $G$ there is a {\bf dual lattice}
given by
$$
\mathcal L^\perp = \{ \alpha \in \widehat G : (\ell, \alpha)=0 \ \mbox{for all} \  \ell \in \mathcal L  \}.
$$
It is well known (see \cite{Rud92}, Theorem 2.1.2) that
\begin{equation} \label{Eq:4-1}
\widehat {\left( G/\mathcal L \right)} \approx \mathcal L^\perp \qquad \mbox{and} \qquad \widehat G / \widehat {\mathcal L} \approx \mathcal L^\perp\,.
\end{equation}

Given two lattices $\mathcal K \subset \mathcal L$ of $G$, the quotient group $\mathcal L /{\mathcal K} \approx (G/\mathcal L)/ (G /\mathcal K) = T_{\mathcal L}/T_{\mathcal K},$ is a finite abelian group since $T_{\mathcal L}$ and $T_{\mathcal K}$ are compact. 


\

We have $\mathcal L^\perp \subset \mathcal K^\perp,$ and therefore $\mathcal K^\perp /{\mathcal L^\perp} $ is also a finite abelian group. In fact, $\mathcal K^\perp /{\mathcal L^\perp}$ and $\mathcal L /{\mathcal K}$ have the same number of elements. To see this, use \eqref{Eq:4-1} with $G=\mathcal L$ and $\mathcal L = \mathcal K$ to deduce $\widehat {\left( \mathcal L/\mathcal K \right)} \approx \widehat{\mathcal L} / {\widehat{\mathcal K}}\,.$ Again by \eqref{Eq:4-1},
$$
\widehat {\left( \mathcal L/\mathcal K \right)} \approx \widehat{\mathcal L} / {\widehat{\mathcal K}} \approx \left(\widehat{G} / {\widehat{\mathcal K}}\right) / \left( \widehat{G} / {\widehat{\mathcal L}}\right) \approx \mathcal K^\perp / \mathcal L^\perp\,.
$$
Since $\mathcal L / {\mathcal K}$ is a finite abelian group, $\widehat {\left( \mathcal L/\mathcal K \right)} \approx \mathcal L / {\mathcal K}$ and the result follows. 


\

If $[\ell]\in \mathcal L/\mathcal K$ and $[\alpha]\in \mathcal K^\perp /\mathcal L^\perp$, the number $ ([\ell ],[\alpha]) := (\ell, \alpha)$ is well defined. Since $\mathcal L/\mathcal K$ and $\mathcal K^\perp /\mathcal L^\perp$ are finite abelian groups, by Theorem 1.2.5 in \cite{Rud92}, 
\begin{equation} \label{Eq:4-2}
\sum_{[\ell]\in \mathcal L/\mathcal K} ([\ell ],[\alpha]) = \begin{cases}
|\mathcal L/\mathcal K| \, &\mbox{ if } \, [\alpha] = [0]\in \mathcal K^\perp /\mathcal L^\perp \\
0\, &\mbox{ if } \, [\alpha] \neq [0]\in \mathcal K^\perp /\mathcal L^\perp 
\end{cases},
\end{equation}

By duality we also have,
\begin{equation} \label{Eq:4-3}
\sum_{[\alpha]\in \mathcal K^\perp/\mathcal L^\perp} ([\ell ],[\alpha]) = \begin{cases}
|\mathcal K^\perp /\mathcal L^\perp|=|\mathcal L/\mathcal K| \, &\mbox{ if } \, [\ell] = [0]\in \mathcal L /\mathcal K \\
0\, &\mbox{ if } \, [\ell] \neq [0]\in \mathcal L /\mathcal K 
\end{cases},
\end{equation}

The {\bf Fourier transform} of $f\in L^1(G, \mu_G)$ is defined by
$$
\widehat f (\alpha) = \int_G f(x)(-x,\alpha)d\mu_G(x)\,, \quad \alpha \in \widehat G\,,
$$
and extends to an unique isometry $\mathcal F(f)= \widehat f$ from $L^2(G, \mu_G)$ into $L^2(\widehat G, \mu_{\widehat G})$, where $\mu_{\widehat G}$ is the Plancherel measure in $\widehat G.$

\

In the sequel we will use the {\bf Poisson Summation Formula} in this situation (see Theorem 5.5.2 in \cite{Rei68}). Let $\mathcal L$ be a lattice in an LCA group $G$ and $F\in C_c(G)$ (the set of continuous functions with compact support on $G$), then
\begin{equation} \label{Eq:4-4}
|T_{\mathcal L}| \sum_{\ell\in \mathcal L} F(\ell) = \sum_{\gamma\in \mathcal L^\perp} \widehat F(\gamma)\,.
\end{equation}

\

\subsection{The Zak transform on LCA groups}  \label{ZakLCA}

Let $\mathcal L$ be a lattice in an LCA group. For $f \in L^1(G)$ the {\bf Zak transform}  of $f$ with respect to the lattice $\mathcal L$ is given by
\begin{equation} \label{Def:4-5}
Z_{\mathcal L} (f)(\alpha ,x) =  |T_{\mathcal L}| \sum_{\ell \in \mathcal L} f(x+\ell) (-\ell, \alpha)\,, \quad \alpha \in \widehat G, \ x\in G¦,.
\end{equation}
It can be extended to an isometric isomorphism from $L^2(G)$ onto $L^2(\widehat{\mathcal L}, L^2(C_{\mathcal L}))$, where $C_{\mathcal L}$ is a measurable set of representatives of $G/\mathcal L$. (For a proof see Proposition 3.3 in \cite{BHP15}.)

\

We list now some properties of the Zak transform just defined. The first one is the following: if $[\alpha_1] = [\alpha_2]$ in $\widehat G / \mathcal L^\perp$, then
\begin{equation} \label{Eq:4-6}
Z_{\mathcal L}(f)(\alpha_1,x) = Z_{\mathcal L}(f)(\alpha_2,x)\,, \quad x\in G\,.
\end{equation}
Indeed, since $[\alpha_1] = [\alpha_2]$ in $\widehat G / \mathcal L^\perp$, there exists $\gamma \in \mathcal L^\perp$ such that $\alpha_1 - \alpha_2 = \gamma.$ Then
\begin{eqnarray*}
	 Z_{\mathcal L}(f)(\alpha_1,x) &=& |T_{\mathcal L}| \sum_{\ell \in \mathcal L} f(x+\ell) (-\ell, \alpha_1)\\
	&=& |T_{\mathcal L}| \sum_{\ell \in \mathcal L} f(x+\ell) (-\ell, \alpha_2+\gamma)\\
	&=& |T_{\mathcal L}| \left(\sum_{\ell \in \mathcal L} f(x+\ell) (-\ell, \alpha_2)\right) (-\ell , \gamma)\\
	&=& Z_{\mathcal L}(f)(\alpha_2,x)\,,
\end{eqnarray*}
since $(-\ell , \gamma)=1$ by definition of $\mathcal L^\perp.$ The second one is related to translations in $G$: if $\ell \in \mathcal L$, then
\begin{equation} \label{Eq:4-7}
Z_{\mathcal L}(f)(\alpha,x-\ell) = (-\ell, \alpha) Z_{\mathcal L}(f)(\alpha,x)\,, \quad x\in G,\ \alpha \in \widehat G.
\end{equation}
In fact, 
\begin{eqnarray*}
	Z_{\mathcal L}(f)(\alpha,x-\ell) &=& |T_{\mathcal L}| \sum_{\ell' \in \mathcal L} f(x-\ell+\ell') (-\ell', \alpha)\\
	&=& |T_{\mathcal L}| \sum_{\ell'' \in \mathcal L} f(x+\ell'') (-\ell'' - \ell, \alpha)\\
	&=& |T_{\mathcal L}| \left(\sum_{\ell'' \in \mathcal L} f(x+\ell'') (-\ell'', \alpha)\right) (-\ell , \alpha)\\
	&=& (-\ell, \alpha) Z_{\mathcal L}(f)(\alpha,x)\,.
\end{eqnarray*}
\begin{Remark}
	It follows from \eqref{Eq:4-7} that if $[x_1]=[x_2]$ in $G/\mathcal L$ and $\alpha \in \mathcal L^\perp$, then $Z_{\mathcal L}(f)(\alpha,x_1) = Z_{\mathcal L}(f)(\alpha,x_2).$
\end{Remark}

\begin{proposition} \label{Pro:4-1}
	Let $\mathcal K \subset \mathcal L$ be two lattices in an LCA group $G$.  For $f\in L^2(G), \ \alpha \in \widehat G,\ x\in G,$
	$$
	Z_{\mathcal K}(f)(\alpha,x) = \sum_{[\beta]\in \mathcal K^\perp / \mathcal L^\perp } Z_{\mathcal L}(f)(\alpha + \beta,x)\,.
	$$
\end{proposition}

\begin{proof}
	Observe that for $[\beta]\in \mathcal K^\perp / \mathcal L^\perp$, $Z_{\mathcal L}(f)(\alpha + \beta, x)$ is well defined by \eqref{Eq:4-6}, that is the formula is independent of the representative chosen in $[\beta].$ 
	By density, it is enough to prove the result for $f\in C_c(G).$ Using definition \eqref{Def:4-5}, 
	\begin{eqnarray*}
		\sum_{[\beta]\in \mathcal K^\perp / \mathcal L^\perp } Z_{\mathcal L}(f)(\alpha + \beta,x) &=& \sum_{[\beta]\in \mathcal K^\perp / \mathcal L^\perp } |T_{\mathcal L}| \sum_{\ell \in \mathcal L} f(x+\ell) (-\ell, \alpha + \beta)\\
		&=& \sum_{\ell \in \mathcal L}  |T_{\mathcal L}| \left(\sum_{[\beta]\in \mathcal K^\perp / \mathcal L^\perp } (-\ell, \beta)\right) f(x+\ell) (-\ell, \alpha)\,.
	\end{eqnarray*}
    By \eqref{Eq:4-3}, $\displaystyle \sum_{[\beta]\in \mathcal K^\perp / \mathcal L^\perp } (-\ell, \beta) = |\mathcal L/ \mathcal K|$ if $[\ell]=[0]$ in $\mathcal L/ \mathcal K$ and equals $0$ ir  $[\ell]\neq [0]$ in $\mathcal L/ \mathcal K$. Since $|\mathcal L/ \mathcal K| = |T_{\mathcal K}|/|T_{\mathcal L}|$ we obtain
    $$
    \sum_{[\beta]\in \mathcal K^\perp / \mathcal L^\perp } Z_{\mathcal L}(f)(\alpha + \beta,x) = \sum_{k \in \mathcal K} |T_{\mathcal K}| f(x+k) (-k, \alpha) = Z_{\mathcal K}(f)(\alpha, x)\,.
    $$
\end{proof}

Recall that $T_x(f)(y) = f(y-x)$ denotes the translation by $x\in G$ of the function $f$ defined in $G$.
\begin{proposition} \label{Pro:4-2}
	Let $\mathcal K \subset \mathcal L$ be two lattices in an LCA group $G$.  For $\ell \in \mathcal L, f\in L^2(G), \ \alpha \in \widehat G,\ x\in G,$
	$$
	Z_{\mathcal K}(T_\ell f)(\alpha,x) = \sum_{[\beta]\in \mathcal K^\perp / \mathcal L^\perp } (-\ell, \alpha + \beta) Z_{\mathcal L}(f)(\alpha + \beta,x).
	$$
\end{proposition}

\begin{proof}
	By density, it is enough to prove the result for $f\in C_c(G).$ Use Proposition \ref{Pro:4-1} and \eqref{Eq:4-7} to obtain
	\begin{eqnarray*}
		 Z_{\mathcal K}(T_\ell f)(\alpha,x) &=& Z_{\mathcal K}( f)(\alpha,x-\ell) \\
		&=& \sum_{[\beta]\in \mathcal K^\perp / \mathcal L^\perp }  Z_{\mathcal L}(f)(\alpha + \beta,x-\ell) \\
		&=& \sum_{[\beta]\in \mathcal K^\perp / \mathcal L^\perp } (-\ell, \alpha + \beta) Z_{\mathcal L}(f)(\alpha + \beta,x).
	\end{eqnarray*}
\end{proof}

As in te case of $G=\R$ we are going to need an expression for the Zak transform of $f\in L^2(G)$ in terms of the Fourier transform of $f$ in $G$. This is possible due to the Poisson Summation Formula \eqref{Eq:4-4}.
\begin{proposition} \label{Pro:4-3}
	For $f\in L^2(G), \ x\in G, \ \alpha \in \widehat G$, and $\mathcal L$ a lattice in $G$, 
	$$
	Z_{\mathcal L}(f)(\alpha,x) = \sum_{\gamma \in \mathcal L^\perp} \widehat f(\alpha + \gamma)(x, \alpha + \gamma).
	$$
\end{proposition}

\begin{proof}
	As before, it is enough to prove the result for $f\in C_c(G).$ Consider the function $F_{\alpha , x} (y) = f(x+y)(-y,\alpha)\,, \ y\in G.$ By the Poisson Summation Formula,
	\begin{equation} \label{Eq:4-8}
	    Z_{\mathcal L} (\alpha, x) =|T_{\mathcal L}| \sum_{\ell \in \mathcal L} F_{\alpha,x}(\ell) = \sum_{\gamma \in \mathcal L^\perp} \widehat{F_{\alpha,x}} (\gamma)\,.
	\end{equation}
	We now compute $\widehat{F_{\alpha,x}} (\gamma):$
	\begin{eqnarray}
		\widehat{F_{\alpha,x}} (\gamma) &=& \int_G F_{\alpha,x} (y) (-y,\gamma)\,d \mu_G (y) \nonumber \\
		&=& \int_G f(x+y) (-y,\alpha) (-y,\gamma)\,d \mu_G (y) \nonumber \\
		&=& \int_G f(x+y) (-y,\alpha + \gamma) \,d \mu_G (y)  \nonumber\\
		&=& \int_G f(z) (x-z,\alpha + \gamma) \,d \mu_G (z) \nonumber\\
		&=& (x, \alpha + \gamma)\ \widehat f (\alpha + \gamma)\,.  \label{Eq:4-9}
	\end{eqnarray}
	The result now follows replacing \eqref{Eq:4-9} in \eqref{Eq:4-8}.
\end{proof}

\

\subsection{Principal invariant spaces in LCA groups} \label{PrincipalLCA}

Let $\mathcal L$ be a lattice in an LCA group $G$. A closed subspace $V$ of $L^2(G)$ is $\mathcal L$ {\bf invariant} if when $f\in V$, $T_\ell (f) \in V$ for all $\ell \in \mathcal L.$ If $\psi \in L^2(G)$, the subspace
$$
\langle \psi \rangle_{\mathcal L} := \overline{span} \{ T_\ell(\psi) : \ell \in \mathcal L  \}\,
$$
 is an $\mathcal L$ invariant subspace of $L^2(G)$ that is called {\bf principal}.
 
\

As in subsection \ref{Invariant}, given two lattices $\mathcal K \subset \mathcal L$ in $G$, we are interested in finding necessary and sufficient conditions on $\psi \in L^2(G)$ for $\langle \psi \rangle_{\mathcal L}$ to be $\mathcal L$ invariant. A preliminary result is the following:

\begin{proposition} \label{Pro:4-4}
	Let $\psi \neq 0$, $\psi\in L^2(G)$, and $\mathcal K \subset \mathcal L$ be two lattices in $G$.  The following are equivalent:
	
	(a) $\langle \psi \rangle_{\mathcal K}$ is $\mathcal L$ invariant.
	
	(b) $T_{\ell}(\psi) \in \langle \psi \rangle_{\mathcal K}$ for all $\ell \in \mathcal L.$
\end{proposition}

\begin{proof}
	$(a) \Rightarrow (b)$ is clear by definition. To prove $(b) \Rightarrow (a)$ let $f\in \langle \psi \rangle_{\mathcal K}.$ We have to show $T_\ell (f) \in \langle \psi \rangle_{\mathcal K}$ for all $\ell \in \mathcal L.$ But
	$$
	T_\ell (f) \in T_\ell (\langle \psi \rangle_{\mathcal K}) \subset \langle T_\ell(\psi) \rangle_{\mathcal K} \subset \langle \psi \rangle_{\mathcal K}\,,
	$$
	since $T_{\ell}(\psi) \in \langle \psi \rangle_{\mathcal K}$ by $(b)$.
\end{proof}

We need now a characterization of $\langle \psi \rangle_{\mathcal K}$ in terms of a multiplier. In the case of $\R$ this was accomplished by means of the Fourier transform. For LCA groups, the right tool is the periodization mapping introduced by H. Helson (see \cite{Hel92}) for the case $G=\mathbb T$ and extended to LCA groups in \cite{CP10}. For $f\in L^2(G)$ the {\bf periodization} mapping of $f$ relative to the lattice $\mathcal K$ is given by
$$
\mathcal T_{\mathcal K}(f)(\alpha) = \{ \widehat f (\alpha + \gamma) \}_{ \gamma \in \mathcal K^\perp}\,, \quad \alpha \in \widehat G\,.
$$
It can be shown (see Proposition 3.3 in \cite{CP10}) that $\mathcal T$ is an isometric isomorphism from $L^2(G)$ onto $L^2(C_{\mathcal K^\perp}\,, \ell^2(\mathcal K^\perp))$, where $C_{\mathcal K^\perp}$ is a measurable section of $\widehat G / \mathcal K^\perp.$ For our purposes we need the following statement of Proposition 3.3 in \cite{CP10} adapted to principal invariant subspaces.

\begin{proposition} \label{Pro:4-5}
	Let $\psi \neq 0$, $\psi\in L^2(G)$, and $\mathcal K$  a lattice in $G$.  
	
	(a) If $f \in \langle \psi \rangle_{\mathcal K}$, there exists a $\mathcal K^\perp$-periodic function $m_f$ on $\widehat G$ such that $ \mathcal T_{\mathcal K}(f)(\alpha) = m_f(\alpha) \mathcal\, T_{\mathcal K}(\psi)(\alpha),\ \alpha \in \widehat G.$
	
	(b) If $m$ is a $\mathcal K^\perp$-periodic function on $\widehat G$ such that $m \mathcal\, T_{\mathcal K}(\psi) \in L^2(C_{\mathcal K^\perp}\,, \ell^2(\mathcal K^\perp))$, the function $f$ defined by  $ \mathcal T_{\mathcal K}(f) = m \mathcal\, T_{\mathcal K}(\psi)$ belongs to $\langle \psi \rangle_{\mathcal K}$.
\end{proposition}

We need a similar result in terms of multipliers of the Zak transform.

\begin{corollary} \label{Cor:4-6}
	Let $\psi \neq 0$, $\psi\in L^2(G)$, and $\mathcal K$  a lattice in $G$.  
	
	(a) If $f \in \langle \psi \rangle_{\mathcal K}$, there exists a $\mathcal K^\perp$-periodic function $m_f$ on $\widehat G$ such that $ Z_{\mathcal K}(f)(\alpha,x) = m_f(\alpha) Z_{\mathcal K}(\alpha, x),\ \alpha \in \widehat G, \ x\in G.$
	
	(b) If $m$ is a $\mathcal K^\perp$-periodic function on $\widehat G$ such that $m \mathcal\, \mathcal T_{\mathcal K}(\psi) \in L^2(C_{\mathcal K^\perp}\,, \ell^2(\mathcal K^\perp))$, the function $f$ defined by  $ Z_{\mathcal K}(f)(\alpha , x) = m(\alpha) Z_{\mathcal K}(\psi)(\alpha, x)$ belongs to $\langle \psi \rangle_{\mathcal K}$.
\end{corollary}

\begin{proof}
	$(a)$ Choose $m_f$ as in part $(a)$ of Proposition \ref{Pro:4-5}. Then, by Proposition \ref{Pro:4-3} for $\mathcal L = \mathcal K$, since $m_f$ is $\mathcal K^\perp$-periodic, we have
	\begin{eqnarray*}
	Z_{\mathcal K}(f)(\alpha , x) &=& \sum_{\gamma \in \mathcal K^\perp} \widehat f(\alpha + \gamma)(x, \alpha + \gamma)\\
	&=& \langle \mathcal T_{\mathcal K}(f)(\alpha) , (x, \cdot)\rangle_{\ell^2(\mathcal K^\perp) } (x,\alpha)  \\
	&=& \langle m_f(\alpha) \mathcal T_{\mathcal K}(\psi)(\alpha) , (x, \cdot)\rangle_{\ell^2(\mathcal K^\perp) } (x,\alpha)  \\
	&=& m_f(\alpha) \langle \mathcal T_{\mathcal K}(\psi)(\alpha) , (x, \cdot)\rangle_{\ell^2(\mathcal K^\perp) } (x,\alpha)   \\
	&=& m_f(\alpha) 	Z_{\mathcal K}(\psi)(\alpha , x)\,.  
	\end{eqnarray*}

$(b)$ If $\alpha \in \widehat G$ and $x \in G$, by Proposition \ref{Pro:4-3} and the $\mathcal K^\perp$-periodicity of $m$, we can write:
\begin{eqnarray*}
	m(\alpha) Z_{\mathcal K}(\psi)(\alpha , x) &=& m(\alpha) \langle \mathcal T_{\mathcal K}(\psi)(\alpha) , (x, \cdot)\rangle_{\ell^2(\mathcal K^\perp) } (x,\alpha) \\
	&=&  \langle m(\alpha) \mathcal T_{\mathcal K}(\psi)(\alpha) , (x, \cdot)\rangle_{\ell^2(\mathcal K^\perp) } (x,\alpha)  \\
	&=& Z_{\mathcal K} (\mathcal T_{\mathcal K}^{-1}(m\mathcal T_{\mathcal K} (\psi)))\,.  
\end{eqnarray*}
By $(b)$, $Z_{\mathcal K} (\mathcal T_{\mathcal K}^{-1}(m\mathcal T_{\mathcal K} (\psi))) = Z_{\mathcal K}(f)$, and since $Z_{\mathcal K}$ is an isometry, we conclude $m \mathcal T_{\mathcal K}(\psi) = \mathcal T_{\mathcal K}(f).$ The result now follows from $(b)$ of Proposition \ref{Pro:4-5}.
\end{proof}

\section{Proof of Theorem \ref{Th:1-2}} \label{ProofLCA}

\subsection{Proof of (a) implies (b) of Theorem \ref{Th:1-2}}

Assume that $\langle \psi \rangle_{\mathcal K}$ is $\mathcal L$ invariant. By Proposition \ref{Pro:4-4}, for every $\ell \in \mathcal L$, we have $T_\ell (\psi) \in \langle \psi \rangle_{\mathcal K}$. By Corollary \ref{Cor:4-6}, there exists a $\mathcal K^\perp$-periodic function $m_\ell$ on $\widehat G$ such that 
\begin{equation} \label{Eq:5-1}
 Z_{\mathcal K}(T_\ell(\psi))(\alpha,x) = m_\ell(\alpha) Z_{\mathcal K}(\psi)(\alpha, x),\quad \alpha \in \widehat G, \ x\in G.
\end{equation}
On the other hand, by Proposition \ref{Pro:4-2}, for $\ell \in \mathcal L$, 
\begin{equation} \label{Eq:5-2}
Z_{\mathcal K}(T_\ell \psi)(\alpha,x) = (-\ell , \alpha) \sum_{[\beta]\in \mathcal K^\perp / \mathcal L^\perp} (-\ell,  \beta) Z_{\mathcal L}(\psi)(\alpha + \beta,x),
\end{equation}
for $\ell \in \mathcal L, \ \alpha \in \widehat G, \ x\in G.$
Define
$$
A_\ell(\alpha, x) := \sum_{[\beta]\in \mathcal K^\perp / \mathcal L^\perp} (-\ell,  \beta) Z_{\mathcal L}(\psi)(\alpha + \beta,x)\,.
$$
We know that for $[\ell] \in \mathcal L /\mathcal K$ and $[\alpha] \in \mathcal K^\perp /\mathcal L^\perp$, $([\ell] , [\alpha]) = (\ell , \alpha)$ is well defined. Also, if $[\ell_1]=[\ell_2]$ in $\mathcal L /\mathcal K$ it can be shown that $A_{\ell_1}(\alpha, x) = A_{\ell_2}(\alpha, x).$ Thus, for 
$[\ell] \in \mathcal L /\mathcal K, \, \alpha \in \widehat G, \, x \in G,$ there is no ambiguity in defining 
\begin{equation} \label{Eq:5-3}
A_{[\ell]}(\alpha, x) := \sum_{[\beta]\in \mathcal K^\perp / \mathcal L^\perp} (-[\ell],  [\beta]) Z_{\mathcal L}(\psi)(\alpha + \beta,x)\,.
\end{equation}
Use the orthogonality relations \eqref{Eq:4-2} to obtain, for $\beta\in  \mathcal K^\perp $, $\alpha\in \widehat G, \ x\in G$,
\begin{equation} \label{Eq:5-4}
Z_{\mathcal L}(\psi)(\alpha + \beta,x)  := \frac{1}{|\mathcal L / \mathcal K|} \sum_{[\ell]\in \mathcal L / \mathcal K} (\ell,  \beta) A_{[\ell]}(\alpha, x)\,.
\end{equation}

\begin{lemma} \label{Lem:5-1}
	If $[\ell_1] + [\ell_2]  = [s_1] + [s_2]$ in $\mathcal L / \mathcal K$, $\alpha \in \widehat G, \, x \in G$, then $$A_{[\ell_1]}(\alpha,x) A_{[\ell_2]}(\alpha,y) = A_{[s_1]}(\alpha,x) A_{[s_2]}(\alpha,y) .$$
\end{lemma}

\begin{proof}
	By \eqref{Eq:5-3}, \eqref{Eq:5-2}, and \eqref{Eq:5-1},
	$$
	A_{[\ell_1]}(\alpha,x) A_{[\ell_2]}(\alpha,y) = (\ell_1 + \ell_2, \alpha) m_{\ell_1}(\alpha) m_{\ell_2}(\alpha) Z_{\mathcal K}( \psi)(\alpha,x) Z_{\mathcal K}( \psi)(\alpha,y)\,.
	$$
	Similarly,
	$$
	A_{[s_1]}(\alpha,x) A_{[s_2]}(\alpha,y) = (s_1 + s_2, \alpha) m_{s_1}(\alpha) m_{s_2}(\alpha) Z_{\mathcal K}( \psi)(\alpha,x) Z_{\mathcal K}( \psi)(\alpha,y)\,.
	$$
	Since $[\ell_1] + [\ell_2]  = [s_1] + [s_2]$ in $ \mathcal L / \mathcal K$, there exists $k\in \mathcal K$ such that $s_1 +s_2 =\ell_1 + \ell_2 +k$. Hence, by \eqref{Eq:4-7} with $\mathcal L = \mathcal K$,
	\begin{eqnarray*}
	    m_{s_1}(\alpha) m_{s_2}(\alpha) Z_{\mathcal K}( \psi)(\alpha,x) &=& Z_{\mathcal  K} (T_{s_1 + s_2}(\psi)) (\alpha, x) \\ &=& Z_{\mathcal  K} (T_k T_{\ell_1 + \ell_2}(\psi)) (\alpha, x) \\
	    &=& (-k,\alpha) Z_{\mathcal  K} ( T_{\ell_1 + \ell_2}(\psi)) (\alpha, x)\\ &=& (-k,\alpha) m_{\ell_1}(\alpha) m_{\ell_2}(\alpha) Z_{\mathcal K}( \psi)(\alpha,x)\,.
    \end{eqnarray*}
    Thus,
    \begin{eqnarray*}
        A_{[s_1]}(\alpha,x) A_{[s_2]}(\alpha,y) &=& (\ell_1 + \ell_2 + k, \alpha) (-k, \alpha) m_{\ell_1}(\alpha) m_{\ell_2}(\alpha) Z_{\mathcal K}( \psi)(\alpha,x) Z_{\mathcal K}( \psi)(\alpha,y)\\
        &=& A_{[\ell_1]}(\alpha,x) A_{[\ell_2]}(\alpha,y) 
    \end{eqnarray*}
    since $(k,\alpha) (- k, \alpha) = |(k,\alpha)|^2 = 1.$
\end{proof}

We continue with the proof of $(a)$ implies $(b)$ of Theorem \ref{Th:1-2}. Choose $[\beta_1] \neq [\beta_2] \in \mathcal K^\perp / \mathcal L^\perp$, $\alpha \in \widehat G, \, x \in G$. By \eqref{Eq:5-4},
\begin{eqnarray*}
& & Z_{\mathcal L}(\psi)(\alpha + \beta_1,x)  \, Z_{\mathcal L}(\psi)(\alpha + \beta_2,y) \\
&=& \frac{1}{|\mathcal L / \mathcal K|^2} \sum_{[\ell]\in \mathcal L / \mathcal K} \sum_{[m]\in \mathcal L / \mathcal K} ([\ell],  [\beta_1]) \, ([m],  [\beta_2]) A_{[\ell]}(\alpha, x)\, A_{[m]}(\alpha, y) \\
&=& \frac{1}{|\mathcal L / \mathcal K|^2} \sum_{[\ell]\in \mathcal L / \mathcal K} \sum_{[s]\in \mathcal L / \mathcal K} ([\ell],  [\beta_1]) \, ([s-\ell],  [\beta_2]) A_{[\ell]}(\alpha, x)\, A_{[s-\ell]}(\alpha, y)\,.
\end{eqnarray*}
Since $[\ell] + [s - \ell] = [s]= [0] + [s]$, by Lemma \ref{Lem:5-1},
\begin{eqnarray*}
	& & Z_{\mathcal L}(\psi)(\alpha + \beta_1,x)  \, Z_{\mathcal L}(\psi)(\alpha + \beta_2,y) \\
	&=& \frac{1}{|\mathcal L / \mathcal K|^2} \sum_{[\ell]\in \mathcal L / \mathcal K} \sum_{[s]\in \mathcal L / \mathcal K} ([\ell],  [\beta_1] - [\beta_2]) \, ([s],  [\beta_2]) A_{[0]}(\alpha, x)\, A_{[s]}(\alpha, y) \\
	&=& \frac{1}{|\mathcal L / \mathcal K|^2} \sum_{[s]\in \mathcal L / \mathcal K} \left( \sum_{[\ell]\in \mathcal L / \mathcal K} ([\ell],  [\beta_1]-[\beta_2]) \right) \, ([s],  [\beta_2]) A_{[0]}(\alpha, x)\, A_{[s]}(\alpha, y)\,.
\end{eqnarray*}
Since, when $[\beta_1] \neq [\beta_2],$ $\displaystyle \sum_{[\ell]\in \mathcal L / \mathcal K} ([\ell],  [\beta_1]-[\beta_2])=0$ by \eqref{Eq:4-2}, the result is established.

\

\subsection{Proof of (b) implies (a) of Theorem \ref{Th:1-2}}

Assume that 
\begin{equation} \label{Eq:5-5}
Z_{\mathcal L}(\psi)(\alpha + \beta_1, x) \, Z_{\mathcal L}(\psi)(\alpha + \beta_2, y) = 0
\end{equation} 
when $[\beta_1] \neq [\beta_2]$ in $\mathcal K^\perp /\mathcal L^\perp$, and a. e. $x,y \in C_{\mathcal L}\,, \alpha \in C_{\mathcal K^\perp}.$ Recall that
\begin{equation} \label{Eq:5-6}
   \bigcup_{\ell\in \mathcal L} C_{\mathcal L} + \ell = G\,, \qquad \mbox{and}  \qquad \bigcup_{\gamma\in \mathcal K^\perp} C_{\mathcal K^\perp} + \gamma = \widehat G \,,
\end{equation}
with disjoint unions. By Proposition \ref{Pro:4-4} and Corollary \ref{Cor:4-6} we have to show that for $\ell \in \mathcal L$ there exists a $\mathcal K^\perp$-periodic function $m_\ell$ defined on $\widehat G$ such that $m_\ell \mathcal\, Z_{\mathcal K}(\psi) \in L^2(C_{\mathcal K^\perp}\,, \ell^2(\mathcal K^\perp))$ and 
\begin{equation} \label{Eq:5-7}
 Z_{\mathcal K}(T_\ell(\psi))(\alpha , x) = m_\ell(\alpha) Z_{\mathcal K}(\psi)(\alpha, x)\,, \quad \alpha\in \widehat G\,, \ x\in G.
\end{equation}

For $x\in G$ and $[\beta]\in \mathcal K^\perp /\mathcal L^\perp$ let
$$
S_\psi^{[\beta]}(x) := \{\alpha \in  C_{\mathcal K^\perp}: Z_{\mathcal L}(\psi)(\alpha + \beta, x) \neq 0 \} \,.
$$
Notice that the definition of $S_\psi^{[\beta]}(x) $ does not depend on the representation chosen for $[\beta]$. Indeed, If $\beta_1 \in [\beta],$ there exists $\gamma \in \mathcal L^\perp$ such that $\beta_1 - \beta = \gamma,$ and since $(-\ell, \gamma)=1$ when $\ell \in \mathcal L$ and $\gamma \in \mathcal L^\perp,$
\begin{eqnarray*}
    Z_{\mathcal L}(\psi)(\alpha+\beta_1, x) &= & \sum_{\ell \in \mathcal L} \psi(x+\ell) (-\ell, \alpha + \beta_1) \\
    &=& \left( \sum_{\ell \in \mathcal L} \psi(x+\ell) (-\ell, \alpha + \beta) \right) (-\ell , \gamma) \\
    &=& Z_{\mathcal L}(\psi)(\alpha+\beta, x)\,.
\end{eqnarray*}

Consider
\begin{equation} \label{Eq:5-8}
S_\psi^{[\beta]} := \bigcup_{x\in C_{\mathcal K}} S_\psi^{[\beta]}(x)\,, \qquad \mbox{and}  \qquad S_\psi := C_{\mathcal K^\perp} \setminus \bigcup_{[\beta]\in \mathcal K^\perp /\mathcal L^\perp} S_\psi^{[\beta]}\,.
\end{equation}
Observe that the union in the left hand side of \eqref{Eq:5-8} is disjoint due to \eqref{Eq:5-5}.

For $\ell \in \mathcal L$ define
\begin{equation} \label{Def:m}
m_\ell (\alpha) := \sum_{[\beta]\in \mathcal K^\perp /\mathcal L^\perp} (-\ell, \alpha) \chi_{S_\psi^{[\beta]}} (\alpha) (-\ell, \beta) + \chi_{S_\psi} (\alpha)\,, \quad \alpha \in  C_{\mathcal K^\perp}\,,
\end{equation}
and extend $m_\ell$ to be $\mathcal K^\perp$-periodic in $\widehat G.$

Notice that, by Proposition \ref{Pro:4-1},
\begin{equation} \label{Eq:5-9}
   Z_{\mathcal K}(\psi)(\alpha,x) = \sum_{[\beta]\in \mathcal K^\perp /\mathcal L^\perp} Z_{\mathcal L}(\psi)(\alpha + \beta,x)\,,
\end{equation}
and, by Proposition \ref{Pro:4-2},
\begin{equation} \label{Eq:5-10}
Z_{\mathcal K}(T_\ell(\psi))(\alpha,x) = \sum_{[\beta]\in \mathcal K^\perp /\mathcal L^\perp} (-\ell, \alpha + \beta) Z_{\mathcal L}(\psi)(\alpha + \beta,x)\,.
\end{equation}

If given $\alpha \in C_{\mathcal K^\perp}$  and  $x \in C_{\mathcal L}\,, $
$Z_{\mathcal L}(\psi)(\alpha + \beta,x) = 0$ for all $[\beta]\in \mathcal K^\perp /\mathcal L^\perp$, then by \eqref{Eq:5-9}, $Z_{\mathcal K}(\psi)(\alpha,x) = 0$, and by \eqref{Eq:5-10}, $Z_{\mathcal K}(T_\ell(\psi))(\alpha,x) = 0$. Therfore, \eqref{Eq:5-7} holds trivially for any value given to $m_\ell$ and in particular for the value given by the definition of $m_\ell$ in \eqref{Def:m}.

If $Z_{\mathcal L}(\psi)(\alpha + \beta,x) \neq 0$ for some $[\beta]\in \mathcal K^\perp /\mathcal L^\perp$, by \eqref{Eq:5-5} and \eqref{Eq:5-9} we have $Z_{\mathcal K}(\psi)(\alpha,x) = Z_{\mathcal L}(\psi)(\alpha + \beta,x),$ and by \eqref{Eq:5-5} and \eqref{Eq:5-10}, $Z_{\mathcal K}(T_\ell(\psi))(\alpha,x) =  (-\ell, \alpha + \beta) Z_{\mathcal L}(\psi)(\alpha + \beta,x)\,.$ In this case $\alpha \in S_\psi^{[\beta]}$ and, by \eqref{Def:m}, $m_\ell (\alpha) = (-\ell, \alpha+\beta),$ so that \eqref{Eq:5-7} also holds. Observe that $|m_\ell(\alpha)|=1$ and since 
$ \mathcal T_{\mathcal K}(\psi) \in L^2(C_{\mathcal K^\perp}\,, \ell^2(\mathcal K^\perp))$, also $m_\ell \mathcal\, \mathcal T_{\mathcal K}(\psi) \in L^2(C_{\mathcal K^\perp}\,, \ell^2(\mathcal K^\perp))$.

Finally, although we have only proved \eqref{Eq:5-7} for $\alpha \in C_{\mathcal K^\perp}$  and  $x \in C_{\mathcal L}\,, $ the quasi-periodicity properties of $Z_{\mathcal K}$ and the periodicity properties of $m_\ell$, together with \eqref{Eq:5-6}, prove the result for all $\alpha \in \widehat G$ and all $x\in G.$

\bibliographystyle{plain}

\vskip 1truemm

\end{document}